\documentclass[a4paper,11pt]{article}

\usepackage[T1]{fontenc}
\usepackage{lmodern}

\usepackage{dsfont}

\usepackage[latin1]{inputenc}
\usepackage{amsmath}
\usepackage{amsthm}
\usepackage{amssymb}
\usepackage{mathrsfs}
\usepackage{graphicx}
\usepackage[all]{xy}
\usepackage{hyperref}

\usepackage{stmaryrd}
\usepackage{caption}

\usepackage{abstract} 

\usepackage{MnSymbol}

\newtheorem{thm}{Theorem}[section]
\newtheorem{cor}[thm]{Corollary}
\newtheorem{claim}[thm]{Claim}
\newtheorem{fact}[thm]{Fact}

\newtheorem{lemma}[thm]{Lemma}
\newtheorem{prop}[thm]{Proposition}

\theoremstyle{definition}
\newtheorem{definition}[thm]{Definition}

\newtheorem{remark}[thm]{Remark}
\newtheorem{question}[thm]{Question}

\def\rquotient#1#2{%
	\makeatletter
	\raise.3ex\hbox{$#1$}/\lower.3ex\hbox{$#2$}%
	\makeatother
}	

\makeatletter
\newcommand{\subjclass}[2][2010]{%
	\let\@oldtitle\@title%
	\gdef\@title{\@oldtitle\footnotetext{#1 \emph{Mathematics subject classification.} #2}}%
}
\newcommand{\keywords}[1]{%
	\let\@@oldtitle\@title%
	\gdef\@title{\@@oldtitle\footnotetext{\emph{Key words and phrases.} #1.}}%
}
\makeatother

\newcommand{\Address}{{
		\bigskip
		\small
		
		\textsc{D\'epartement de Math\'ematiques B\^atiment 307, Facult\'e des Sciences d'Orsay, Universit\'e Paris-Sud, F-91405 Orsay Cedex, France.}\par\nopagebreak
		\textit{E-mail address}: \texttt{anthony.genevois@math.u-psud.fr}
		
}}

\title{Negative curvature in automorphism groups of one-ended hyperbolic groups}
\date{\today}
\author{Anthony Genevois}


\begin{document}

\maketitle

\begin{abstract}
In this article, we show that some negative curvature may survive when taking the automorphism group of a finitely generated group. More precisely, we prove that the automorphism group $\mathrm{Aut}(G)$ of a one-ended hyperbolic group $G$ turns out to be acylindrically hyperbolic. As a consequence, given a group $H$ and a morphism $\varphi : H \to \mathrm{Aut}(G)$, we deduce that the semidirect product $G \rtimes_\varphi H$ is acylindrically hyperbolic if and only if $\mathrm{ker}(H \overset{\varphi}{\to} \mathrm{Aut}(G) \to \mathrm{Out}(G))$ is finite. 
\end{abstract}

\tableofcontents

\section{Introduction}

Finding out some information about the automorphism group is one the most natural questions we can ask about a given group. It is also a particularly difficult one. The main reason is that, typically, it is much more difficult to work with the automorphism group than the group itself. As a consequence, being able to read properties of the automorphism group directly from the group may be quite useful.

Taking the point of view of geometric group theory, one of the most efficient methods to study a group is to exhibit some phenomena of negative curvature. So a natural but vague question is the following: does the negatively-curved geometry of a group survive when taking its automorphism group?

Such a question does not seem to be so unreasonable. For instance, the study of (outer) automorphism groups of free groups and of surface groups, two of the most studied families of groups in geometric group theory, is fundamentally based on the negatively-curved geometries of free groups and surface groups. Inspired by these two examples, we ask the following more precise question:

\begin{question}\label{question}
Is the automorphism group of a finitely generated acylindrically hyperbolic group acylindrically hyperbolic as well?
\end{question}

The class of acylindrically hyperbolic groups, as defined in \cite{OsinAcyl}, encompasses many interesting families of groups, including in particular non-elementary hyperbolic and relatively hyperbolic groups, most 3-manifold groups, groups of deficiency at least two, many small cancellation groups, the Cremona group of birational transformations of the complex projective plane, many Artin groups, and of course the two examples mentioned above, namely outer automorphism groups of free groups (of rank at least two) and mapping class groups of (non-exceptional) surfaces. Nevertheless, being acylindrically hyperbolic provides valuable information on the group, as many aspects of the theory of hyperbolic and relatively hyperbolic groups can be generalised in the context of acylindrical hyperbolicity. These include various algebraic, model-theoretic, and analytic properties; small cancellation theory; and group theoretic Dehn surgery. We refer to \cite{OsinSurvey} and references therein for more information. 

\medskip 
This article is dedicated to the proof of the following statement, which provides a partial positive answer to Question \ref{question}:

\begin{thm}\label{thm:main}
The automorphism group of any one-ended hyperbolic group is acylindrically hyperbolic.
\end{thm}

\noindent
Let us mention two applications of this statement.

\begin{cor}
Let $X$ be a compact $n$-dimensional Riemannian manifold of negative curvature, $n \geq 2$. Then the automorphism group $\mathrm{Aut}(\pi_1(X))$ is acylindrically hyperbolic.
\end{cor}

\begin{proof}
The fundamental group $\pi_1(X)$ is a hyperbolic group whose boundary is an $(n-1)$-sphere. Because $n \geq 2$, it must be one-ended. So Theorem \ref{thm:main} applies.
\end{proof}

\noindent
In our next application, \emph{random groups} are defined following \cite{ChampetierStat}. 

\begin{cor}
Almost surely, the automorphism group of a random group is acylindrically hyperbolic.
\end{cor}

\begin{proof}
Following \cite[Th\'eor\`eme 4.18]{ChampetierStat}, almost surely a random group is a hyperbolic group whose boundary is a Menger curve. Such a group must be one-ended. So Theorem~\ref{thm:main} applies.
\end{proof}

We emphasize that, although the structure of outer automorphism groups of one-ended hyperbolic groups is pretty well-understood (see \cite[Corollary 4.7]{HypIso}), it is not sufficient to deduce our theorem. In fact, outer automorphism groups are rarely acylindrically hyperbolic. Also, many of the properties which follow from the acylindrical hyperbolicity cannot be deduced from the study of the outer automorphism group. For instance, if the outer automorphism group is virtually free abelian (which happens in particular if the JSJ decomposition of our hyperbolic group contains only rigid pieces, so that Dehn twists generate a free abelian subgroups of finite index in the outer automorphism group; see \cite[Corollary 4.7]{HypIso} for more details), it is not clear how to deduce that the automorphism group is SQ-universal or that it satisfies the property $P_{\text{naive}}$; but it can be done as soon as we know that the automorphism group is acylindrically hyperbolic (see \cite[Theorem 2.33]{DGO} and \cite{Pnaive} respectively). 

\medskip
Our proof of Theorem \ref{thm:main} is based on JSJ decompositions of hyperbolic groups as constructed in \cite{BowditchJSJ}. Such a decomposition exists if our group $G$ is not virtually a surface group and if $\mathrm{Out}(G)$ is infinite. (Notice that, if $\mathrm{Out}(G)$ is finite, then $\mathrm{Aut}(G)$ turns out to be commensurable to $G$ as $\mathrm{Inn}(G) \simeq G/Z(G)$ has finite index in $\mathrm{Aut}(G)$; consequently, $\mathrm{Aut}(G)$ must be hyperbolic. And if $G$ is virtually a surface group, the acylindrical hyperbolicity of $\mathrm{Aut}(G)$ follows easily from what we know about mapping class groups of surfaces; see Corollary \ref{cor:MCGaut}.) Next, the argument goes as follows:
\begin{itemize}
	\item The starting point of our argument is the well-known fact that the canonicity of the JSJ decomposition of a one-ended hyperbolic group $G$ implies that the automorphism group $\mathrm{Aut}(G)$ naturally acts on the associated Bass-Serre tree $T$. 
	\item Next, we observe that $G$ acts on the JSJ tree $T$ with WPD elements. By applying a small cancellation theorem of \cite{DGO}, it follows that there exists some $g \in G$ such that the normal closure $\llangle g \rrangle$ is free and intersects trivially vertex- and edge-stabilisers of $T$. 
	\item The key point of the argument is to show that there exists a finitely generated free subgroup $H \leq \llangle g \rrangle$ which is not elliptic in any splitting of $G$ with virtually cyclic edge-groups. As a consequence of Whitehead's work \cite{Whitehead}, we know that there exists some $h \in H$ such that $H$ does not split freely relatively to $\langle h \rangle$.
	\item By looking at the action $\mathrm{Aut}(G) \curvearrowright T$, it turns out that, if the inner automorphism $\iota(h)$ is not WPD with respect to $\mathrm{Aut}(G) \curvearrowright T$, then there must exist infinitely many pairwise non-conjugate automorphisms sending a fixed power of $h$ to the same element. 
	\item By applying Paulin's construction \cite{PaulinOut}, one gets an action of $G$ on some real tree, namely one of the asymptotic cones of $G$, with respect to which $g$ is elliptic. As a consequence of Rips' theory, as exposed in \cite{GuirardelTrees}, it follows that $G$ splits relatively to $\langle h \rangle$ over a virtually cyclic subgroup.
	\item The consequence is that $H$ must split freely relatively to $\langle h \rangle$, which is impossible by definition of $h$. Therefore, the inner automorphism $\iota(h)$ turns out to be a WPD element with respect to $\mathrm{Aut}(G) \curvearrowright T$, proving the acylindrical hyperbolicity of $\mathrm{Aut}(G)$. 
\end{itemize}

As it can be seen, two points are fundamental in our strategy: the JSJ decomposition has to be canonical, and a sequence of pairwise non-conjugate automorphisms has to lead to a splitting of the group. We expect that our arguments still hold when these two ingredients are present, leading to the acylindrical hyperbolicity of the automorphism group. Many JSJ decompositions of many different kinds of groups can be found in the literature, but very few are known to be canonical. And the combination of Paulin's construction with Rips' theory is essentially the only known method to construct splittings from sequences of automorphisms, but most of the time it is a strategy which is difficult to apply outside the world of hyperbolic groups where asymptotic cones are not real trees. Nevertheless, part of these arguments has been applied to some right-angled Artin groups in \cite{AcylAutRAAG}, and we expect that the same strategy will allow us to extend Theorem \ref{thm:main} to toral relatively hyperbolic groups.

\medskip
Interestingly, Theorem \ref{thm:main} provides non-trivial information on extensions of hyperbolic groups. In fact, in full generality, the acylindrical hyperbolicity of the automorphism group implies that many extensions of the group must be acylindrically hyperbolic as well. More precisely:

\begin{thm}\label{thm:extension}
Let $G$ be a group whose center is finite and whose automorphism group is acylindrically hyperbolic. Fix a group $H$ and a morphism $\varphi : H \to \mathrm{Aut}(G)$. The semidirect product $G \rtimes_\varphi H$ is acylindrically hyperbolic if and only if $$K:= \mathrm{ker} \left( H \overset{\varphi}{\to} \mathrm{Aut}(G) \to \mathrm{Out}(G) \right)$$ is a finite subgroup of $H$.
\end{thm}

Because non-elementary hyperbolic groups have finite centers, the following statement is an immediate consequence of Theorems \ref{thm:extension} and \ref{thm:main}:

\begin{cor}
Let $G$ be a one-ended hyperbolic group. Fix a group $H$ and a morphism $\varphi : H \to \mathrm{Aut}(G)$. The semidirect product $G \rtimes_\varphi H$ is acylindrically hyperbolic if and only~if $$K:= \mathrm{ker} \left( H \overset{\varphi}{\to} \mathrm{Aut}(G) \to \mathrm{Out}(G) \right)$$ is a finite subgroup of $H$.
\end{cor}

\noindent
In the specific case of surface groups, one gets:

\begin{cor}
Let $S$ be a closed orientable surface of genus $\geq 2$. Fix a point $p \in S$, a group $H$ and morphism $\varphi : H \to \mathrm{MCG}^{\pm}(S, p)$. The semidirect product $\pi_1(S,p) \rtimes_\varphi H$ is acylindrically hyperbolic if and only if $\mathrm{ker} \left( H \overset{\varphi}{\to} \mathrm{MCG}^{\pm}(S,p) \to \mathrm{MCG}^{\pm}(S) \right)$ is finite.
\end{cor}

In this statement, $\mathrm{MCG}^\pm(S)$ (resp. $\mathrm{MCG}^\pm (S,p)$) refers to the \emph{extended mapping class group}, i.e., the group of diffeomorphisms of $S$ (resp. the group of diffeomorphisms of $S$ fixing $p$) up to isotopy (resp. up to isotopy fixing $p$). An orientation of $S$ is not assumed to be preserved. There clearly exists a morphism $\psi : \mathrm{MCG}^\pm(S,p) \to \mathrm{Aut}(\pi_1(S,p))$, and we denote by abuse of notation $G\rtimes_\varphi H$ the semidirect product $G\rtimes_{\psi \circ \varphi} H$. Finally, the morphism $f : \mathrm{MCG}^\pm(S,p) \to \mathrm{MCG}^\pm (S)$ is the natural morphism which just ``forgets'' the point $p$.

In the specific case of cyclic extensions, i.e., when $H$ is infinite cyclic and acts on $G$ through an automorphism $\varphi \in \mathrm{Aut}(G)$, Theorem \ref{thm:extension} shows that $G \rtimes_\varphi H$ is acylindrically hyperbolic if and only if $\varphi$ has infinite order in $\mathrm{Out}(G)$. A similar statement holds for free groups \cite{FreeByCyclicAcyl} (up to finite index) and for some right-angled Artin groups \cite{AcylAutRAAG}. We do not know any counterexample to this statement in the context of arbitrary finitely generated acylindrically hyperbolic groups. (Such a counterexample would provide a negative answer to Question \ref{question}.)

\medskip
Let us conclude this introduction with a few remarks. First, it is worth noticing that Question \ref{question} has a negative answer if the group is not assumed to be finitely generated, as shown by \cite[Remark 4.10]{GM}. Next, Theorem \ref{thm:main} does not apply to all non-elementary hyperbolic groups: what about infinitely-ended hyperbolic groups? In this case, the hyperbolic group splits over a finite subgroup, but we cannot expect to make the automorphism group act on the associated Bass-Serre tree since automorphism groups of free products may satisfy strong fixed-point properties \cite{AutUnivCoxFAn, AutFreeFAn, AutFreeT, AutFreeTGeneral} including Serre's property FA \cite{FAmcg, AutFreeProductFA}. So a different approach is needed here. Nevertheless, we expect that the automorphism group of any finitely generated group (not necessarily hyperbolic) splitting over a finite subgroup is acylindrically hyperbolic (or virtually cyclic if the initial group was virtually cyclic). 

\paragraph{Organisation of the paper.} In Section \ref{section:JSJ}, we collect the few statements about JSJ decompositions of hyperbolic groups which will be needed in the paper. In Section \ref{section:WPD}, we state and prove a sufficient condition for an element of a hyperbolic group which is loxodromic in the JSJ tree to define a WPD element of the automorphism group via the inner automorphism associated to it. And in Section \ref{section:Paulin}, we construct a relative splitting of the group when this condition fails. Finally, Section \ref{section:main} is dedicated to the proof of Theorem \ref{thm:main}, and Section \ref{section:extension} to Theorem \ref{thm:extension}.

\paragraph{Acknowledgment.} I am grateful to the anonymous referee for his comments and suggestions on the first version of this article, which led to several improvements.

\section{JSJ decompositions of hyperbolic groups}\label{section:JSJ}

\noindent
Our study of automorphism groups of one-ended hyperbolic groups is based on the notion of \emph{JSJ decompositions}. Initially introduced by Sela in \cite{SelaJSJ}, we use the construction given by Bowditch in \cite{BowditchJSJ}. A simplified version of \cite[Theorem 0.1]{BowditchJSJ} is:

\begin{thm}\label{thm:JSJ}
Let $G$ be a one-ended hyperbolic group which is not virtually a surface group. Then there is a canonical splitting of $G$ as a finite graph of groups such that each edge-group is virtually infinite cyclic and such that there exist two types of vertex-groups:
\begin{itemize}
	\item vertex-groups of type 1 are virtually free;
	\item and vertex-groups of type 2 are quasiconvex subgroups not of type 1.
\end{itemize}
Moreover, two vertex-groups of type 2 cannot be adjacent, and any virtually cyclic subgroup on which $G$ splits can be conjugate into an edge-group or a vertex-group of type~1. 
\end{thm}

\noindent
(In Bowditch's original statement, the hyperbolic group is required not to be a cocompact Fuchsian group and to have a locally connected boundary. However, as a consequence of \cite[Theorem 5.4]{BoundaryHyp} and the related discussion, a hyperbolic group is a cocompact Fuschian group if and only if it is virtually a surface group; and, as discussed in the introduction of \cite{BowditchJSJ}, the boundary of a one-ended hyperbolic group is always locally connected. Another difference is that \cite[Theorem 0.1]{BowditchJSJ} provides three types of vertex-groups: (i) two-ended subgroups (or equivalently, virtually $\mathbb{Z}$ subgroups), (ii) maximal ``hanging Fuchsian'' subgroups, (iii) non-elementary quasiconvex subgroups not of type (i). As discussed after \cite[Theorem 0.1]{BowditchJSJ}, vertex-groups of type (ii) are virtually free. Our vertex-groups of type 1 (resp. of type 2) correspond to vertex-groups of type (i) or (ii) (resp. of type (iii)) in Bowditch's terminology.)

\medskip \noindent
We refer to this decomposition of $G$ as its \emph{JSJ decomposition}, and to the associated Bass-Serre tree as its \emph{JSJ tree}. The uniqueness of the decomposition follows from the fact that Bowditch's construction is based on the boundary of $G$. See the discussion in \cite[Section 6]{BowditchJSJ} for more information. As a consequence, the JSJ tree turns out to be preserved by automorphisms. More precisely, we have:

\begin{prop}
Let $G$ be a one-ended hyperbolic group which is not virtually a surface group. Denote by $T$ the JSJ tree of $G$. The center $Z(G)$ of $G$ is included into the kernel of the action $G \curvearrowright T$, and, if we identify canonically the quotient $G/Z(G)$ with the subgroup of inner automorphisms $\mathrm{Inn}(G) \leq \mathrm{Aut}(G)$, then the action $G/Z(G) \curvearrowright T$ extends to an action $\mathrm{Aut}(G) \curvearrowright T$ via:
$$\left\{ \begin{array}{ccc} \mathrm{Aut}(G) & \to & \mathrm{Isom}(T) \\ \varphi & \mapsto & (x \mapsto \text{vertex whose stabiliser is $\varphi(\mathrm{stab}(x))$}) \end{array} \right. .$$
\end{prop}

\noindent
The fact that the action of the center $Z(G)$ of $G$ on the JSJ tree is trivial follows from the observation that $Z(G)$ acts trivially on the boundary of $G$. (More generally, the kernel of the action $G \curvearrowright \partial G$ turns out to coincide with the unique maximal finite normal subgroup of $G$.)

\medskip \noindent
We emphasize that the JSJ decomposition may be trivial. Indeed, taking a one-ended and torsion-free hyperbolic group with a finite outer automorphism group, the JSJ decomposition must be trivial since such a group does not split over a cyclic subgroup according to \cite[Theorem 1.4]{LevittOut}. However, as shown by our next lemma, the JSJ decomposition turns out to be non-trivial as soon as the outer automorphism group is infinite. (Notice that, as first observed in \cite{MillerJSJ}, the outer automorphism group may be finite and the JSJ decomposition non-trivial.)

\begin{lemma}\label{lem:JSJnontrivial}
Let $G$ be a one-ended hyperbolic group which is not virtually a surface group. If $\mathrm{Out}(G)$ is infinite, then the JSJ decomposition of $G$ is non-trivial. As a consequence, the action of $G$ on its JSJ tree admits loxodromic isometries.
\end{lemma}

\begin{proof}
The first remark is that, as a consequence of \cite[Theorem 5.28]{BowditchJSJ}, the action $G \curvearrowright T$ is minimal. Consequently, in order to deduce that the JSJ decomposition is not trivial, it is sufficient to verify that it cannot be reduced to a single vertex-group under our assumptions. 

\medskip \noindent
First of all, because our group is one-ended, the JSJ decomposition cannot be reduced to a single vertex-group of type 1. Next, according to \cite{OutHyp}, as $\mathrm{Out}(G)$ is infinite necessarily $G$ has to split over a virtually cyclic subgroup. But such a subgroup must be included into a vertex-group of type 1 of the JSJ decomposition (up to conjugation). Indeed, this subgroup can be conjugate into a vertex-group of type 1 or an edge-group; but no two vertex-groups of type 2 are adjacent, so an edge-group always embeds into a vertex-group of type 1. It follows that the JSJ decomposition cannot be reduced to a single vertex-group of type~2.

\medskip \noindent
Thus, we have proved the first assertion of our lemma. The second assertion follows from the fact that a finitely generated group acting on a tree either fixes a point or contains a loxodromic isometry (see for instance \cite[Corollary 3 page 65, Proposition 25 page~63]{MR1954121}).
\end{proof}

\noindent
In our argument, the following observation will be fundamental:

\begin{prop}\label{prop:WPDinG}
Let $G$ be a one-ended hyperbolic group which is not virtually a surface group. Any element $g \in G$ defining a loxodromic isometry of the JSJ tree $T$ must be WPD with respect to $G \curvearrowright T$. 
\end{prop}

\noindent
Recall that, given a group $G$ acting on metric space $X$ by isometries, an element $g \in G$ is WPD if, for every $x \in X$ and every $\epsilon>0$, there exists some $N \geq 1$ such that the set
$$\{ h \in G \mid d(x,hx) \leq \epsilon \ \text{and} \ d(g^Nx,hg^Nx) \leq \epsilon\}$$
is finite. WPD elements are fundamental in the theory of acylindrically hyperbolic groups since we can define acylindrically hyperbolic groups as non-virtually cyclic groups admitting actions on hyperbolic spaces with at least one WPD element \cite{OsinAcyl}. 

\medskip \noindent
In the case of a tree, WPD elements can be characterised in an easier way. In the sequel, we will often use the following statement without mentioning it.

\begin{lemma}
Let $G$ be a group acting on a simplicial tree $T$, and $g \in G$ a loxodromic isometry. Then $g$ is WPD if and only if there exist two points $x,y \in \mathrm{axis}(g)$ such that the intersection $\mathrm{stab}(x) \cap \mathrm{stab}(y)$ is finite. 
\end{lemma}

\noindent
The fact that the latter condition is sufficient to deduce that an element is WPD is proved in \cite[Corollary 4.3]{MinasyanOsin}. The converse is an immediate consequence of the definition. Now, let us go back to Proposition \ref{prop:WPDinG}.

\begin{proof}[Proof of Proposition \ref{prop:WPDinG}.]
Let $T$ denote the JSJ tree of $G$. Suppose that $g \in G$ is loxodromic and fix an edge $e$ in its axis. If $\mathrm{stab}(e) \cap \mathrm{stab}(ge)$ is infinite, then $g$ belongs to the commensurator of $\mathrm{stab}(e)$. But, since $\mathrm{stab}(e)$ is virtually cyclic, it has finite-index in its commensurator, so there must exist some power $s \geq 1$ such that $g^s$ belongs to $\mathrm{stab}(e)$. This contradicts the fact that $g$ is a loxodromic isometry of $T$, concluding the proof of our proposition. 
\end{proof}

\section{Inner automorphisms as generalised loxodromic elements}\label{section:WPD}

\noindent
We saw in the previous section that the automorphism group $\mathrm{Aut}(G)$ of our hyperbolic group $G$ acts naturally on the JSJ tree of $G$. A natural strategy to prove that $\mathrm{Aut}(G)$ is acylindrically hyperbolic is to show that it contains WPD elements with respect to this action. The most natural attempt in this direction is to start with a WPD element $g$ of $G$ and to look at the corresponding inner automorphism $\iota(g)$. The next statement provides a sufficient criterion to determine when such an inner automorphism defines a WPD element with respect to the action of $\mathrm{Aut}(G)$ on the JSJ tree. 

\begin{prop}\label{prop:ifnotWPD}
Let $G$ be a one-ended hyperbolic group which is not virtually a surface group. Let $T$ denote the JSJ tree of $G$. Suppose that $g \in G$ is WPD with respect to $G \curvearrowright T$ but that $\iota(g^s)$ is not WPD with respect to $\mathrm{Aut}(G) \curvearrowright T$ for every $s \geq 1$. Then there exist $s \geq 1$ and infinitely many pairwise non-conjugate automorphisms $\varphi_1, \varphi_2, \ldots \in \mathrm{Aut}(G)$ such that $\varphi_i(g^s)= \varphi_j(g^s)$ for every $i,j \geq 1$. 
\end{prop}

\begin{proof}
Fix some element $g \in G$ which is WPD with respect to $G \curvearrowright T$, and let $\gamma \subset T$ denote its axis. Fix an edge $e \subset \gamma$ and let $Z$ denote its $G$-stabiliser. Since $Z$ has finite index in its normaliser $N(Z)$, the orbit $N(Z) \cdot e$ must be finite. Let $C$ denote the convex hull of this orbit; it is a finite subtree. It follows from the fact that $g$ is WPD with respect to $G \curvearrowright T$ that there exists some $s \geq 1$ such that $\mathrm{stab}_G(C) \cap \mathrm{stab}_G(g^sC)$ is finite. As
$$N(Z) \cap g^sN(Z)g^{-s} \subset \mathrm{stab}_G(C) \cap \mathrm{stab}_G(g^sC),$$
we deduce that $N(Z) \cap g^sN(Z)g^{-s}$ must be finite. Now, we know by assumption that $\iota(g^s)$ is not WPD with respect to $\mathrm{Aut}(G) \curvearrowright T$, so the intersection
$$P :=\mathrm{stab}_{\mathrm{Aut}(G)}(e) \cap \mathrm{stab}_{\mathrm{Aut}(G)}(g^{2s}e)$$
must be infinite. First, since
$$P \cap \mathrm{Inn}(G) = \{ \iota(h) \mid h \in \mathrm{stab}_G(e) \cap \mathrm{stab}_G(g^{2s}e) \} \subset \{ \iota(h) \mid h \in \mathrm{stab}_G(e) \cap \mathrm{stab}_G(g^{s}e) \},$$
we notice that $P \cap \mathrm{Inn}(G)$ is finite, so that $P$ has infinite image in $\mathrm{Out}(G)$. Fix a sequence of automorphisms $\varphi_1, \varphi_2, \ldots \in P$ which are pairwise non-conjugate. 

\medskip \noindent
Next, fix some $\varphi \in P$. Because $\varphi$ belongs to $\mathrm{stab}_{\mathrm{Aut}(G)}(e)$, we know that $\varphi(Z)=Z$; because $\varphi$ belongs to $\mathrm{stab}_{\mathrm{Aut}(G)}(g^se)$, we know that
$$g^sZg^{-s}  =  \varphi(g^sZg^{-s})= \varphi(g^s) \varphi(Z) \varphi(g^{-s}) = \varphi(g^s) Z \varphi(g^{-s}),$$
hence $\varphi(g^s)=g^sn$ for some $n \in N(Z)$; and because $\varphi$ belongs to $\mathrm{stab}_{\mathrm{Aut}(G)}(g^{2s}e)$, we know that
$$\begin{array}{lcl} g^{2s}Zg^{-2s} & = & \varphi(g^{2s}Zg^{-2s}) = \varphi(g^s) \varphi(g^s)\varphi(Z) \varphi(g^s)^{-1} \varphi(g^s)^{-1} \\ \\ & = & g^sng^snZn^{-1}g^{-s}n^{-1}g^{-s}= g^sng^sZg^{-s}n^{-1}g^{-s}, \end{array}$$
hence $n \in g^sN(Z)g^{-s}$. Thus, we have proved that
$$P \subset \{ \varphi \in \mathrm{Aut}(G) \mid \varphi(g^s) \in g^s \left( N(Z) \cap g^s N(Z)g^{-s} \right) \}.$$
But we know that the intersection $N(Z) \cap g^s N(Z)g^{-s}$ is finite, so we can find a subsequence in $\varphi_1, \varphi_2, \ldots \in P$ all of whose automorphisms send $g^s$ to the same element.
\end{proof}

\section{Relative splittings from Paulin's construction}\label{section:Paulin}

\noindent
In this section, we are interested in understanding what happens when Proposition \ref{prop:ifnotWPD} applies, or more precisely, when there exists an element of the group which is fixed by infinitely many pairwise non-conjugate automorphisms. Our main result in this direction is the following: 

\begin{prop}\label{prop:RelativeSplitting}
Let $G$ be a hyperbolic group and $g \in G$ an infinite-order element. Suppose that there exist infinitely many pairwise non-conjugate automorphisms $\varphi_1, \varphi_2, \ldots \in \mathrm{Aut}(G)$ such that $\varphi_i(g)= \varphi_j(g)$ for every $i,j \geq 1$. Then $G$ splits relatively to $\langle g \rangle$ over a virtually cyclic subgroup.
\end{prop}

\noindent
Recall that a group $G$ splits relatively to a subgroup $H$ if $H$ can be conjugate into one of the factors of the splitting. Proposition \ref{prop:RelativeSplitting} is clearly inspired by \cite[Corollary 1.3]{OutHyp}, and similarly our proof is based on Paulin's construction \cite{PaulinOut}. First of all, we need to recall basic definitions related to asymptotic cones of groups.

\paragraph{Asymptotic cones.} An \emph{ultrafilter} $\omega$ over a set $S$ is a collection of subsets of $S$ satisfying the following conditions:
\begin{itemize}
	\item $\emptyset \notin \omega$ and $S \in \omega$; 
	\item for every $A,B \in \omega$, $A \cap B \in \omega$;
	\item for every $A \subset S$, either $A\in \omega$ or $A^c \in \omega$.
\end{itemize}
Basically, an ultrafilter may be thought of as a labelling of the subsets of $S$ as ``small'' (if they do not belong to $\omega$) or ``big'' (if they belong to $\omega$). More formally, notice that the map
$$\left\{ \begin{array}{ccc} \mathfrak{P}(S) & \to & \{ 0,1\} \\ A & \mapsto & \left\{ \begin{array}{cl} 0 & \text{if $A \notin \omega$} \\ 1 & \text{if $A \in \omega$} \end{array} \right. \end{array} \right.$$
defines a finitely additive measure on $S$. 

\medskip \noindent
The easiest example of an ultrafilter is the following. Fixing some $s \in S$, set $\omega= \{ A \subset S \mid s \in A \}$. Such an ultrafilter is called \emph{principal}. The existence of non-principal ultrafilters is assured by Zorn's lemma; see \cite[Section 3.1]{KapovichLeebCones} for a brief explanation.

\medskip \noindent
Now, fix a metric space $(X,d)$, a non-principal ultrafilter $\omega$ over $\mathbb{N}$, a \emph{scaling sequence} $\epsilon= (\epsilon_n)$ satisfying $\epsilon_n \to 0$, and a sequence of basepoints $o=(o_n) \in X^{\mathbb{N}}$. A sequence $(r_n) \in \mathbb{R}^{\mathbb{N}}$ is \emph{$\omega$-bounded} if there exists some $M \geq 0$ such that $\{ n \in \mathbb{N} \mid |r_n| \leq M \} \in \omega$ (i.e., if $|r_n| \leq M$ for ``$\omega$-almost all $n$''). Set
$$B(X,\epsilon,o) = \{ (x_n) \in X^{\mathbb{N}} \mid \text{$(\epsilon_n \cdot d(x_n,o_n))$ is $\omega$-bounded} \}.$$
We may define a pseudo-distance on $B(X,\epsilon,o)$ as follows. First, we say that a sequence $(r_n) \in \mathbb{R}^{\mathbb{N}}$ \emph{$\omega$-converges} to a real $r \in \mathbb{R}$ if, for every $\epsilon>0$, $\{ n \in \mathbb{N} \mid |r_n-r| \leq \epsilon \} \in \omega$. If so, we write $r= \lim\limits_\omega r_n$. It is worth noticing that an $\omega$-bounded sequence of $\mathbb{R}^\mathbb{N}$ always $\omega$-converges; see \cite[Section 3.1]{KapovichLeebCones} for more details. Then, our pseudo-distance is
$$\left\{ \begin{array}{ccc} B(X,\epsilon,o)^2 & \to & [0,+ \infty) \\ (x,y) & \mapsto & \lim\limits_\omega \epsilon_n \cdot d(x_n,y_n) \end{array} \right..$$
Notice that the previous $\omega$-limit always exists since the sequence under consideration is $\omega$-bounded. 

\begin{definition}
The \emph{asymptotic cone} $\mathrm{Cone}_\omega(X,\epsilon,o)$ of $X$ is the metric space obtained by quotienting $B(X,\epsilon,o)$ by the relation: $(x_n) \sim (y_n)$ if $d\left( (x_n),(y_n) \right)=0$.  
\end{definition}

\noindent
The picture to keep in mind is that $(X, \epsilon_n \cdot d)$ is sequence of spaces we get from $X$ by ``zooming out'', and the asymptotic cone if the ``limit'' of this sequence. Roughly speaking, the asymptotic cones of a metric space are asymptotic pictures of the space. For instance, any asymptotic cone of $\mathbb{Z}^2$, thought of as the infinite grid in the plane, is isometric to $\mathbb{R}^2$ endowed with the $\ell^1$-metric; and the asymptotic cones of a simplicial tree (and more generally of any Gromov-hyperbolic space) are real trees.

\paragraph{Paulin's construction.} This paragraph is dedicated to the main construction of \cite{PaulinOut}, allowing us to construct an action of a group on one of its asymptotic cones thanks to a sequence of pairwise non-conjugate automorphisms. As our language is different (but equivalent), we give a self-contained exposition of the construction below. 

\medskip \noindent
Let $G$ be a non-trivial finitely generated group with a fixed generating set $S$ and let $\varphi_1, \varphi_2, \ldots \in \mathrm{Aut}(G)$ be a collection of automorphisms. The goal is to construct a non-trivial action of $G$ on one of its asymptotic cones from the sequence of twisted actions
$$\left\{ \begin{array}{ccc} G & \to & \mathrm{Isom}(G) \\ g & \to & \left( h \mapsto \varphi_n(g) \cdot h \right) \end{array} \right., \ n \geq 1.$$
For every $n \geq 1$, set $\lambda_n = \min\limits_{x \in G} \max\limits_{s \in S} d(x, \varphi_n(s) \cdot x)$; notice that $\lambda_n \geq 1$ since $\varphi_n(s) \cdot x \neq x$ for every $n \geq 1$ and every non-trivial $s \in S$. We suppose that $\lambda_n \underset{n \to + \infty}{\longrightarrow}+ \infty$. Now, fixing some non-principal ultrafilter $\omega$ over $\mathbb{N}$ and some sequence $o=(o_n) \in G^{\mathbb{N}}$ satisfying, for every $n \geq 1$, the equality $\max\limits_{s \in S} d(o_n, \varphi_n(s) \cdot o_n) =\lambda_n$, notice that the map
$$\left\{ \begin{array}{ccc} G & \to & \mathrm{Isom}(\mathrm{Cone}(G)) \\ g & \mapsto & \left( (x_n) \mapsto (\varphi_n(g) \cdot x_n) \right) \end{array} \right.$$
defines an action by isometries on $\mathrm{Cone}(G):= \mathrm{Cone}_{\omega}(G, (1/\lambda_n),o)$. The only fact to verify is that, if $g \in G$ and if $(x_n)$ defines a point of $\mathrm{Cone}(G)$, then so does $(\varphi_n(g) \cdot x_n)$. Writing $g$ as product of generators $s_1 \cdots s_r$ of minimal length, we have
$$\begin{array}{lcl} \displaystyle \frac{1}{\lambda_n} d(o_n, \varphi_n(g) \cdot x_n) & \leq & \displaystyle \frac{1}{\lambda_n} \sum\limits_{i=1}^r d(o_n, \varphi_n(s_i) \cdot x_n) \\ \\ & \leq & \displaystyle \frac{1}{\lambda_n} \sum\limits_{i=1}^r \left( d(o_n, \varphi_n(s_i) \cdot o_n) + d(o_n,x_n) \right) \\ \\ & \leq & \displaystyle \|g \|_S \left( 1+ \frac{1}{\lambda_n} d(o_n,x_n) \right) \end{array}$$
Consequently, if the sequence $(d(o_n,x_n))$ is $\omega$-bounded, i.e., if $(x_n)$ defines a point of $\mathrm{Cone}(G)$, then the sequence $(d(o_n,\varphi_n(g) \cdot x_n))$ is $\omega$-bounded as well, i.e., $(\varphi_n(g) \cdot x_n)$ also defines a point of $\mathrm{Cone}(G)$.

\begin{fact}\label{fact:PaulinFixedPoint}
The action $G \curvearrowright \mathrm{Cone}(G)$ does not fix a point.
\end{fact}

\noindent
Suppose that $G$ fixes a point $(x_n)$ of $\mathrm{Cone}(G)$. Then, for $\omega$-almost all $n$ and all $s \in S$, the inequality
$$\frac{1}{\lambda_n} d(x_n, \varphi_n(s) \cdot x_n) \leq \frac{1}{2}$$
holds, hence
$$\lambda_n \leq \max\limits_{s \in S} d(x_n, \varphi_n(s) \cdot x_n) \leq \lambda_n/2,$$
which is impossible.

\medskip \noindent
The conclusion is that we can associate a fixed-point free action of $G$ on one of its asymptotic cones from an infinite collection of automorphisms, provided that our sequence $(\lambda_n)$ tends to infinity. So the natural question is now: when does it happen? 

\begin{fact}\label{fact:lambdainfinity}
If the automorphisms $\varphi_1, \varphi_2, \ldots$ of $G$ are pairwise non-conjugate, then the equality $\lim\limits_{\omega} \lambda_n = + \infty$ holds. 
\end{fact}

\noindent
Suppose that the $\omega$-limit of $(\lambda_n)$ is not infinite. So there exists a subsequence $(\lambda_{\sigma(n)})$ which is bounded above by some constant $R$. So, for every $n \geq 1$ and every $s \in S$, one has
$$d \left( 1, o_{\sigma(n)}^{-1} \varphi_{\sigma(n)}(s) o_{\sigma(n)} \right) = d \left( o_{\sigma(n)}, \varphi_{\sigma(n)}(s) o_{\sigma(n)} \right) \leq \lambda_n \leq R,$$
i.e., $o_{\sigma(n)}^{-1} \varphi_{\sigma(n)}(s) o_{\sigma(n)} \in B(1,R)$. Since the ball $B(1,R)$ is finite, it follows from the pigeonhole principle that there exist two distinct $m,n \geq 1$ such that
$$o_{\sigma(n)}^{-1} \varphi_{\sigma(n)}(s) o_{\sigma(n)} = o_{\sigma(m)}^{-1} \varphi_{\sigma(m)}(s) o_{\sigma(m)}$$
for every $s \in S$. It follows that $\varphi_{\sigma(n)}$ and $\varphi_{\sigma(m)}$ are conjugate. This concludes the proof of our fact.

\medskip \noindent
So the point of Paulin's construction is that we can associate a fixed-point free action of $G$ on one of its asymptotic cones from an infinite subset of $\mathrm{Out}(G)$.

\paragraph{Rips' machinery.} Typically, when looking at an action on an asymptotic cone, the objective is to get an action on a real tree in order to construct, thanks to \emph{Rips' machinery}, an action on a simplicial tree, or equivalently according to Bass-Serre theory, a splitting of the corresponding group. This strategy is illustrated by the following statement, which is a special case of \cite[Corollary 5.2]{GuirardelTrees}. We recall that a group $G$ splits relatively to a subgroup $H$ if $G$ decomposes as an HNN extension or an amalgamated product such that $H$ is included into a factor, or equivalently such that $H$ fixes a vertex in the associated Bass-Serre tree.

\begin{thm}\label{thm:Ripsmachin}
Let $G$ be a finitely generated group acting fixed-point freely and minimally on a real tree $T$, and $H \subset G$ a subgroup fixing a point of $T$. Suppose that:
\begin{itemize}
	\item arc-stabilisers are finitely generated;
	\item a non-decreasing sequence of arc-stabilisers stabilises;
	\item for every arc $I$, there does not exist $g \in G$ satisfying $g \cdot \mathrm{stab}(I) \cdot g^{-1} \subsetneq \mathrm{stab}(I)$.
\end{itemize}
Assuming that $T$ is not a line, then $G$ splits relatively to $H$ over a cyclic-by-(arc-stabiliser) subgroup.
\end{thm}

\paragraph{Proof of the main proposition.} We are now ready to prove Proposition \ref{prop:RelativeSplitting}. The strategy is to find an action on a real tree thanks to Paulin's construction and next to apply Theorem \ref{thm:Ripsmachin}. 

\begin{proof}[Proof of Proposition \ref{prop:RelativeSplitting}.]
Fix a finite generating set $S$ of $G$, a non-principal ultrafilter $\omega$ over $\mathbb{N}$, and, for every $n \geq 1$, set $\lambda_n = \min\limits_{x \in G} \max\limits_{s \in S} d(x, \varphi_n(s) \cdot x)$. According to Fact~\ref{fact:lambdainfinity}, one has $\lim\limits_{\omega} \lambda_n = + \infty$. It follows from Paulin's construction explained above that there exists some $o=(o_n) \in G^{\mathbb{N}}$ such that $G$ acts on the asymptotic cones $\mathrm{Cone}(G)= \mathrm{Cone}_\omega(G,(1/\lambda_n),o)$ via
$$\left\{ \begin{array}{ccc} G & \to & \mathrm{Isom}(\mathrm{Cone}(G)) \\ g & \mapsto & ((x_n) \mapsto (\varphi_n(g) \cdot x_n)) \end{array} \right..$$
Moreover, this action is fixed-point free according to Fact \ref{fact:PaulinFixedPoint}. Notice that $\mathrm{Cone}(G)$ is a real tree since $G$ is hyperbolic (see for instance \cite[Example 7.30]{Roe}).

\begin{claim}
The element $g$ fixes a point in $\mathrm{Cone}(G)$.
\end{claim}

\noindent
Let $h \in G$ be such that $\varphi_i(g)=h$ for every $i \geq 1$. Because $h$ has infinite order, it admits a quasi-axis in $G$. More precisely, if $\|h \|$ denotes $\min \{ d_S(x,hx) \mid x \in G\}$ and if $a \in G$ is such that $d(a,ha)= \|h\|$, then the orbit $\langle h \rangle \cdot a$ defines a quasi-geodesic line $\gamma$ \cite[Lemma~6.5 in Chapter 10]{GeomHypGroupes} on which $h$ acts by translations of length $\|h\|$. It is worth noticing that, although $\|h^k\|$ may differ from $k \|h\|$ for some $k \geq 1$, the difference $| \|h^k\|-k\|h\| |$ is bounded above by a constant which depends only on the hyperbolicity constant of $G$, say $\delta$. Indeed, setting $[h] = \lim\limits_{i \to + \infty} \frac{1}{i} d(x,h^ix)$ for some arbitrary basepoint $x \in G$ (the limit always exists and does not depend on the choice of $x$ according to \cite[Proposition~6.1 in Chapter 10]{GeomHypGroupes}), it is clear that $[h^k]=k \cdot [h]$ and moreover the difference $|[h^k] - \|h^k\| |$ is bounded above by a constant which only depends on $\delta$ according to \cite[Proposition~6.4 in Chapter 10]{GeomHypGroupes}, which proves our claim. By fixing some $k \geq 1$ sufficiently large compared to $\delta$, there exists some constant $D \geq 0$ such that
$$\left| d(x,h^kx)- k \| h \| - 2d(x,\gamma) \right| \leq D$$
for every $x \in G$. As a consequence, for every $n \geq 1$, one has 
$$d_S \left( o_n, \varphi_n(g^k) \cdot o_n \right)=d_S \left( o_n,h^k \cdot o_n \right) \geq k \|h \| +2d_S(o_n,p_n) - D$$
where $p_n$ denotes a point of $\gamma$ minimising the distance to $o_n$. Since $g^k \cdot o = (\varphi_n(g^k) \cdot o_n)$ defines a point of $\mathrm{Cone}(G)$, it follows from the previous inequality that the sequence $(d_S(o_n,p_n)/\lambda_n)$ is $\omega$-bounded, so that $p=(p_n)$ defines a point of $\mathrm{Cone}(G)$. We have
$$\begin{array}{lcl} d_{\mathrm{Cone}}(p,g^k \cdot p) & = & \displaystyle \lim\limits_{\omega} \frac{1}{\lambda_n} d_S(p_n, \varphi_n(g^k) \cdot p_n) = \lim\limits_{\omega} \frac{1}{\lambda_n} d_S(p_n, h^k \cdot p_n) \\ \\ & \leq & \displaystyle \lim\limits_{\omega} \frac{1}{\lambda_n} \left( k \| h \| +D \right) =0. \end{array}$$
Thus, we have proved that $g^k$ fixes a point of $\mathrm{Cone}(G)$, which implies that $g$ has to fix a point of $\mathrm{Cone}(G)$ as well, concluding the proof of our claim.

\begin{claim}
Arc-stabilisers in $\mathrm{Cone}(G)$ are virtually cyclic.
\end{claim}

\noindent
For a proof of this claim, we refer to \cite[Proposition 2.4]{PaulinOut} whose arguments can be easily adapted to the language of asymptotic cones. 

\medskip \noindent
Let $T \subset \mathrm{Cone}(G)$ be a subtree on which $G$ acts minimally. Notice that $T$ cannot be a line since otherwise the commutator subgroup of $G$ would be abelian, implying that $G$ is solvable, and thus contradicting the fact that $G$ is a non-elementary hyperbolic group. Moreover:
\begin{itemize}
	\item A non-decreasing sequence $Z_1 \subset Z_2 \subset \cdots$ of virtually cyclic subgroups of $G$ stabilises. If the $Z_i$'s are all finite, the conclusion follows from the fact that the cardinality of a finite subgroup in a hyperbolic group is bounded above by a constant which depends only on the hyperbolicity constant of the group (see \cite[Corollary 2.2.B]{GromovHyp} or \cite{FiniteHypSub}). Consequently, up to extracting a subsequence, we may suppose without loss of generality that the $Z_i$'s are all infinite. Notice that $Z_1$ must have finite index in $Z_i$ for every $i \geq 1$, so that the $Z_i$'s are all contained into the commensurator $$\mathrm{Com}(Z_1)= \left\{ g \in G \mid [Z_1: Z_1 \cap gZ_1g^{-1}], [gZ_1g^{-1}: Z_1 \cap gZ_1g^{-1}] < + \infty \right\}$$ of $Z_1$. But, as a quasiconvex subgroup, $Z_1$ has finite index in its commensurator \cite{QuasiconvexHyp}, so that $Z_1,Z_2, \ldots$ defines a collection of subgroups of $\mathrm{Com}(Z_1)$, which is virtually cyclic, all of indices $\leq [\mathrm{Com}(Z_1):Z_1]<+ \infty$. Because there exist only finitely many such subgroups, the desired conclusion follows. 
	\item If $Z \leq G$ is virtually cyclic and $g \in G$, then $g Z g^{-1} \subset Z$ implies $gZg^{-1}=Z$. If $Z$ is finite, the conclusion is clear as $Z$ and $gZg^{-1}$ have the same cardinality. Now assume that $Z$ is infinite cyclic. Notice that $g$ belongs to the commensurator of $Z$. But, as a quasiconvex subgroup, $Z$ has finite index in its commensurator \cite{QuasiconvexHyp} so that there exists some $k \geq 1$ such that $g^k \in Z$. From $$Z=g^kZg^{-k} \subset g^{k-1}Z g^{-k+1} \subset g^{k-2} Z g^{-k+2} \subset \cdots \subset gZg^{-1} \subset Z$$ the desired equality follows.
\end{itemize}
Consequently, Theorem \ref{thm:Ripsmachin} applies, and we conclude that $G$ must split relatively to $\langle g \rangle$ over a virtually cyclic subgroup. 
\end{proof}

\section{Proof of the main theorem}\label{section:main}

\noindent
Our last section is dedicated to the proof of Theorem \ref{thm:main}, namely:

\begin{thm}\label{thm:mainstrong}
If $G$ is a one-ended hyperbolic group, then $\mathrm{Aut}(G)$ is acylindrically hyperbolic.
\end{thm}

\noindent
We need to distinguish two cases, depending on whether $G$ is virtually a surface group or not. In the former case, the desired conclusion essentially follows from the next general observation:

\begin{lemma}\label{lem:AutAcylGen}
Let $\mathcal{H}$ be a collection of finitely generated groups such that:
\begin{itemize}
	\item if $H \in \mathcal{G}$ then every finite-index subgroup of $H$ belongs to $\mathcal{H}$;
	\item if $H \in \mathcal{H}$, then $\mathrm{Aut}(H)$ acts non-elementarily and acylindrically on a hyperbolic space with at least one inner automorphism which is a loxodromic isometry. 
\end{itemize}
If $G$ is a hyperbolic group containing a finite-index subgroup in $\mathcal{H}$, then $\mathrm{Aut}(G)$ is acylindrically hyperbolic or virtually cyclic.
\end{lemma}

\begin{proof}
Let $G$ be a group and $H \leq G$ a finite-index subgroup which belongs to $\mathcal{H}$. Assume that $\mathrm{Aut}(G)$ is not virtually cyclic. Up to replacing $H$ with $\bigcap\limits_{\varphi \in \mathrm{Aut}(G)} \varphi(H)$, which is also a finite-index subgroup of $G$ as $G$ contains only finitely many subgroups of a given index and which also belongs to $\mathcal{H}$ because it has finite index in $H$ for the same reason, we may suppose without loss of generality that $H$ is a characteristic subgroup of $G$. As a consequence, the restriction map $\varphi \mapsto \varphi_{|H}$ induces a morphism $\mathrm{Aut}(G) \to \mathrm{Aut}(H)$. Notice that, as a consequence of the next claim, the kernel of this morphism, which coincides with the collection $\mathrm{Aut}(G,H)$ of the automorphisms of $G$ fixing $H$ pointwise, turns out to be finite.

\begin{claim}\label{claim:finitelymany}
Let $G$ be a hyperbolic group and $H \leq G$ a finite-index subgroup. Then $\mathrm{Aut}(G,H)$ is finite.
\end{claim}

\noindent
Let us explain how to conclude the proof of our lemma by assuming this claim.

\medskip \noindent
Because being acylindrically hyperbolic is stable under quotients with finite kernels (according to \cite[Lemma 1]{ErratumMO}), $\mathrm{Aut}(G)$ is acylindrically hyperbolic if so is its image in $\mathrm{Aut}(H)$, which coincides with the subgroup $\mathrm{Aut}(H < G) \leq \mathrm{Aut}(H)$ of the automorphisms of $H$ which extends to $G$. Notice that $\mathrm{Aut}(H<G)$ contains $\mathrm{Inn}(H)$. Therefore, the acylindrical action of $\mathrm{Aut}(H)$ given by assumption provides an acylindrical action of $\mathrm{Aut}(H<G)$ on a hyperbolic space with at least one loxodromic isometry. It follows that $\mathrm{Aut}(H<G)$ is either acylindrically hyperbolic or virtually cyclic, concluding the proof.

\medskip \noindent
Now, let us turn to the proof of Claim \ref{claim:finitelymany}. More precisely, we want to prove that the finite-index subgroup of $\mathrm{Aut}(G,H)$ acting trivially on $G/H$ is finite. Fix a system of representatives $a_1, \ldots, a_n \in G$ so that $G = Ha_1 \sqcup \cdots \sqcup Ha_n$. (Notice that $n= [G:H]$.) Without loss of generality, we may suppose that, for every $1 \leq i \leq n$, the element $a_i$ has infinite order. Indeed, assume that $a_i$ has finite order. Fix two points $\alpha, \beta \in \partial G$ such that $ \beta \notin \{ \alpha, a_i \cdot \alpha \}$. (Notice that, as $\mathrm{Aut}(H)$ is acylindrically hyperbolic, $H$ has to be non-elementary. So the existence of our points $\alpha$ and $\beta$ is justified by the fact that $\partial G$ is infinite.) Because the collection of pairs of points fixed by the infinite-order elements of $H$ is dense in $\partial G \times \partial G$ \cite[Corollary 8.2.G]{GromovHyp}, there exists an infinite-order element $h \in H$ whose fixed points at infinity are arbitrarily close to $\alpha$ and $\xi$. As a consequence of \cite[Theorem 3.1]{Simul}, $h^ma_i$ must have infinite order for some $m \geq 1$. Therefore, replacing $a_i$ with $h^ma_i$ leads to the desired conclusion.

\medskip \noindent
So let $\varphi \in \mathrm{Aut}(G,H)$ be an automorphism acting trivially on $G/H$. In other words, $\varphi$ is an automorphism of $G$ which fixes pointwise $H$ and which stabilises each coset of $H$. As a consequence, there exists a map $\sigma : G \to H$ which is constant to $1$ on $H$ and such that $\varphi(g)=g \sigma(g)$ for every $g \in G$. Notice that, for every $a,b \in G$, one has
$$ab \sigma(ab) = \varphi(ab)= \varphi(a) \varphi(b) = a \sigma(a) b \sigma(b),$$
hence $\sigma(ab)= b^{-1} \sigma(a) b \sigma(b)$. The latter equality will referred to in the sequel as the \emph{cocycle equality}. 

\medskip \noindent
For every $1 \leq i \leq n$, it follows from the cocycle equality and from the fact that $\sigma$ is constant to $1$ on $H$ that
$$\varphi(ha_i)= ha_i \sigma(ha_i) = h a_i a_i^{-1} \sigma(h) a_i \sigma(a_i)= ha_i \sigma(a_i)$$
for every $h \in H$. Therefore, $\varphi$ is uniquely determined by the values $\sigma(a_1), \ldots, \sigma(a_n)$. As a consequence, if we show that, for every $1 \leq i \leq n$, the number of possible values for $\sigma(a_i)$ is bounded above by a constant which does not depend on $\varphi$, then the proof of our claim will be complete.

\medskip \noindent
So fix some $a \in \{a_1, \ldots, a_n\}$. Recall that $n= [G:H]$, so $a^n$ belongs to $H$. It follows from the cocycle equality that
$$\sigma(a^n)= \sigma(a \cdot a^{n-1})= a^{-n+1} \sigma(a) a^{n-1} \sigma(a^{n-1}).$$
By induction, we deduce that $\sigma(a^n)= a^{-n+1} (\sigma(a)a)^{n-1} \sigma(a)$. Indeed, the equality clearly holds for $n=1$, and if we assume that it also holds for $n-1$ (where $n \geq 2$) then
$$\begin{array}{lcl} \sigma(a^n) & = & a^{-n+1}\sigma(a)a^{n-1} \sigma(a^{n-1}) = a^{-n+1} \sigma(a) a^{n-1} \cdot a^{-n+2} (\sigma(a)a)^{n-2} \sigma(a) \\ \\ & = & a^{-n+1} \sigma(a)a (\sigma(a)a)^{n-2} \sigma(a) =  a^{-n+1}(\sigma(a)a)^{n-1} \sigma(a). \end{array}$$
But $\sigma$ is constant to $1$ on $H$, so $\sigma(a^n)=1$. Therefore, we have the equality $a^n = (\sigma(a)a)^n$. But $G$ is hyperbolic and $a$ has infinite order, so $a^n$ may have only finitely many $n$th roots. (Indeed, if $r_1, \ldots, r_k$ are $n$th roots of an infinite-order element $g$ and if $k$ is greater than the index of $\langle g \rangle$ inside the centraliser of $g$, then there must exist $1 \leq i<j \leq k$ such that $r_i^{-1}r_j \in \langle g \rangle$ as $r_1, \ldots, r_k$ clearly commute with $g$. By fixing some $s \in \mathbb{Z}$ such that $r_i^{-1}r_j =g^s$, one gets
$$g = r_j^n = (r_ig^s)^n = r_i^n g^{ns}=g^{ns+1},$$
hence $g^{ns}=1$. We conclude that $s=0$ since $g$ has infinite order, i.e., $r_i=r_j$.) 
Hence there are finitely many possible choices for $\sigma(a)$. Thus, we have proved that the number of possible values for the $\sigma(a_i)$'s is bounded above by a constant which depends only on our representatives $a_1, \ldots, a_n$ and on the index $[G:H]$. The proof of our claim is complete. 
\end{proof}

\begin{remark}
It is not true in full generality that $\mathrm{Aut}(G,H)$ must be finite if $H$ is a finite index subgroup of $G$. Indeed, let $H$ be a finitely generated group whose center contains infinitely many elements of order a fixed integer $p \geq 2$. (Many such groups exist as a consequence of \cite{Ould}. Explicit examples are Abels' matrix groups \cite{Abels}.) Notice that, if $\sigma : \mathbb{Z}/p \mathbb{Z} \to Z(H)$ is a morphism to the center of $H$, then 
$$\hat{\sigma} : (h,m) \mapsto (h \sigma(m),m)$$
defines an automorphism of $G:= H \oplus \mathbb{Z}/p \mathbb{Z}$ which belongs to $\mathrm{Aut}(G,H)$. As there exist infinitely many morphisms $\mathbb{Z}/p \mathbb{Z} \to Z(H)$ by construction, it follows that $\mathrm{Aut}(G,H)$ is infinite.
\end{remark}

\noindent
As shown by the next corollary, Lemma \ref{lem:AutAcylGen} applies to one-ended hyperbolic groups which are virtually surface groups. But we also expect other applications, for instance for virtually free groups.

\begin{cor}\label{cor:MCGaut}
Let $G$ be a one-ended hyperbolic group. If $G$ is virtually a surface group, then $\mathrm{Aut}(G)$ is acylindrically hyperbolic.
\end{cor}

\begin{proof}
It is sufficient to show that Lemma \ref{lem:AutAcylGen} applies to 
$$\mathcal{H}= \{ \pi_1(S) \mid \text{$S$ closed surface of genus $\geq 2$}\}.$$ 
More precisely, given a surface $S$ of genus $g \geq 2$, we claim that $\mathrm{Aut}(\pi_1(S,p))$ acts non-elementarily and acylindrically on a hyperbolic space with at least one inner automorphism which is a loxodromic isometry. First of all, consider the following commutative diagram:
\[
\xymatrix{
1 \ar[r] & \mathrm{Inn}(\pi_1(S,p)) \ar[r] & \mathrm{Aut}(\pi_1(S,p)) \ar[r] & \mathrm{Out}(\pi_1(S,p)) \ar[r] & 1 \\
1  \ar[r] & \pi_1(S,p) \ar[u]^\simeq \ar[r] & \mathrm{Mod}^\pm(S,p) \ar[u]^\simeq \ar[r] & \mathrm{Mod}^\pm (S) \ar[u]^\simeq \ar[r] & 1
}
\]
\noindent
Here, $\mathrm{MCG}^\pm(S)$ (resp. $\mathrm{MCG}^\pm (S,p)$) refers to the \emph{extended mapping class group}, i.e., the group of diffeomorphisms of $S$ (resp. the group of diffeomorphisms of $S$ fixing $p$) up to isotopy (resp. up to isotopy fixing $p$). An orientation of $S$ is not assumed to be preserved. The second row is the so-called \emph{Birman exact sequence}. We refer to \cite[Page 235]{BookMCG} for more details. 

\medskip \noindent
According to \cite{BowditchAcylMCG}, the mapping class group $\mathrm{Mod}^\pm(S,p)$ admits a non-elementary and acylindrical action on a hyperbolic space, namely the curve graph of the surface, and the elements of $\mathrm{Mod}^\pm(S,p)$ inducing loxodromic isometries of this space coincide with its \emph{pseudo-Anosov} elements. (See \cite{MCGacyl} and references therein.) As a consequence of \cite[Theorem 2 page 252]{KraMCG}, if $\gamma \subset S$ is a filling closed curve passing through $p$, then the image of the corresponding element of $\pi_1(S,p)$ inside $\mathrm{Mod}^\pm (S,p)$ is pseudo-Anosov. This concludes the proof of our corollary.
\end{proof}

\noindent
Before turning to the proof of Theorem \ref{thm:mainstrong}, let us state a few results which will be needed. The first result we need is \cite[Theorem 8.7]{DGO}. 

\begin{thm}\label{thm:mainDGO}
Let $G$ be a group acting on a hyperbolic space $X$. If $g \in G$ is WPD, then there exists some $s \geq 1$ such that the normal closure $\llangle g^s \rrangle$ is a free subgroup all of whose non-trivial elements are loxodromic isometries of $X$. 
\end{thm}

\noindent
Next, we will need some information on free splittings of finitely generated free groups. The following statement is essentially a consequence of Whitehead's work \cite{Whitehead}. We refer to \cite[Lemma 5.3]{CashenManning} for a proof. 

\begin{prop}\label{prop:FreeSplit}
Let $F$ be a finitely generated free group. There exists some $g \in F$ such that $F$ does not split freely relatively to $\langle g \rangle$.
\end{prop}

\noindent
We are finally ready to prove Theorem \ref{thm:mainstrong}.

\begin{proof}[Proof of Theorem \ref{thm:mainstrong}.]
Let $G$ be a one-ended hyperbolic. If $G$ is virtually a surface group, then Corollary \ref{cor:MCGaut} provides the desired conclusion. From now on, we suppose that $G$ is not virtually a surface group. Let $T$ denote its JSJ tree. 

\medskip \noindent
Suppose first that no element of $G$ defines a loxodromic isometry of $T$. Lemma \ref{lem:JSJnontrivial} implies that $\mathrm{Out}(G)$ must be finite. Consequently, $\mathrm{Aut}(G)$ contains the subgroup of inner automorphisms as a finite-index subgroup isomorphic to the quotient of $G$ by its center. Since the center of a non-elementary hyperbolic group must be finite, we conclude that $\mathrm{Aut}(G)$ must be a non-elementary hyperbolic group as well. In particular, it has to be acylindrically hyperbolic.

\medskip \noindent
From now on, suppose that $G$ contains elements which are loxodromic isometries of $T$. As a consequence of Proposition \ref{prop:WPDinG}, there exists some $g \in G$ which is WPD with respect to the action $G \curvearrowright T$, and, as a consequence of Theorem \ref{thm:mainDGO}, up to replacing $g$ with one of its powers we may suppose that $\llangle g \rrangle$ is a free subgroup intersecting trivially vertex- and edge-stabilisers of $T$. Fix a finite generating set $S$ of $G$, and consider the subgroup
$$H = \left\langle xgx^{-1}, \ x \in S \cup \{1\} \right\rangle \leq \llangle g \rrangle.$$
According to Proposition \ref{prop:FreeSplit}, there exists some $h \in H$ such that $H$ does not split freely relatively to $\langle h \rangle$. 

\medskip \noindent
Fix a splitting of $G$ over a virtually cyclic subgroup, and let $T'$ denote the associated Bass-Serre tree. 

\begin{claim}
The subgroup $H$ does not fix a point of $T'$.
\end{claim}

\noindent
If $h$ is a loxodromic isometry of $T'$, then the conclusion is clear. So suppose that $h$ is elliptic. Notice that $h$ cannot stabilise an edge since it follows from Theorem \ref{thm:JSJ} that $H$ intersects trivially any virtually cyclic subgroup of $G$ over which $G$ splits. So the fixed-point set of $h$ has to be reduced to a single vertex, say $p \in T'$. Now, we distinguish two cases.

\medskip \noindent
Suppose first that there exists $x \in S$ such that $x$ is a loxodromic isometry of $T'$. Then $h$ and $xhx^{-1}$ are two elliptic isometries whose fixed-point sets are disjoint, so that it follows from \cite[Proposition I.6.26]{MR1954121} that $h \cdot xhx^{-1}$ is a loxodromic element of $T'$. Because this element belongs to $H$, we conclude that $H$ does not fix a point of $T'$.

\medskip \noindent
Now, suppose that any element of $S$ is elliptic in $T'$. Since $G$ does not fix a point of $T'$ and that $S$ generates $G$, we know that there exists $x \in S$ which does not fix the unique vertex $p$ fixed by $h$. Let $q \in T'$ denote a vertex fixed by $x$. Notice that, because $h$ does not stabilise an edge, the vertex $p$ separates $q$ and $h \cdot q$. As a consequence, if the fixed-point sets of $x$ and $hx^{-1}h^{-1}$ intersect, then $p$ must be fixed by either $x$ or $hx^{-1}h^{-1}$. We already know that $x$ does not fix $p$. But, in the latter case, we deduce that
$$x^{-1} \cdot p = h^{-1} \cdot hx^{-1}h^{-1} \cdot h \cdot p= h^{-1} \cdot h^{-1}x^{-1}h^{-1} \cdot p = h^{-1} \cdot p = p,$$
which is impossible. Therefore, the fixed-point sets of $x$ and $hx^{-1}h^{-1}$ must be disjoint. It follows from \cite[Proposition I.6.26]{MR1954121} that $x \cdot hx^{-1}h^{-1}$ is a loxodromic element of $T'$. Because this element belongs to $H$, we conclude that $H$ does not fix a point of $T'$.

\medskip \noindent
This concludes the proof of our claim, i.e., $H$ acts fixed-point freely on $T'$.

\medskip \noindent
Now, notice that $H$ has trivial edge-stabilisers in $T'$. Indeed, according to Theorem \ref{thm:JSJ}, any edge-stabiliser of $T'$ must fix a point in $T$, so that the conclusion follows from the fact that $H$ intersects trivially vertex- and edge-stabilisers of $T$. The conclusion is that $H$ acts on $T'$ fixed-point freely and with trivial edge-stabilisers. By definition of the element $h \in H$, we deduce that $h$ is a loxodromic isometry of $T'$.

\medskip \noindent
Thus, we have proved that $h$ cannot be elliptic in the Bass-Serre tree associated to a splitting of $G$ over a virtually cylic subgroup. The same conclusion holds for any non-trivial power of $h$. Therefore, for any $s \geq 1$, $G$ does not split relatively to $\langle h^s \rangle$ over a virtually cyclic subgroup. Notice that, as a consequence of Proposition \ref{prop:WPDinG}, $h$ is WPD with respect to $G \curvearrowright T$, so that Proposition \ref{prop:ifnotWPD} applies. We deduce from Proposition \ref{prop:RelativeSplitting} that the inner automorphism $\iota(h)$ is WPD with respect to the action $\mathrm{Aut}(G) \curvearrowright T$.

\medskip \noindent
Therefore, $\mathrm{Aut}(G)$ is either acylindrically hyperbolic or virtually cyclic. But $\mathrm{Inn}(G) \simeq G/Z(G)$ is not virtually cyclic since $G$ is one-ended, so $\mathrm{Aut}(G)$ cannot be virtually cyclic. We conclude that $\mathrm{Aut}(G)$ is acylindrically hyperbolic as desired.
\end{proof}

\section{Applications to extensions of hyperbolic groups}\label{section:extension}

\noindent
This section is dedicated to Theorem \ref{thm:extension}, namely:

\begin{thm}\label{thm:ExtensionGen}
Let $G$ be a group whose center is finite and whose automorphism group is acylindrically hyperbolic. Fix a group $H$ and a morphism $\varphi : H \to \mathrm{Aut}(H)$. The semidirect product $G \rtimes_\varphi H$ is acylindrically hyperbolic if and only if $$K:= \mathrm{ker} \left( H \overset{\varphi}{\to} \mathrm{Aut}(H) \to \mathrm{Out}(G) \right)$$ is a finite subgroup of $H$.
\end{thm}

\begin{proof}
Assume first that $G \rtimes_\varphi H$ is acylindrically hyperbolic. Because $G$ is an infinite normal subgroup, it follows from \cite[Lemma 7.2]{OsinAcyl} that it must contain a generalised loxodromic element, say $g \in G$. (Recall from \cite{OsinAcyl} that, given a group $\Gamma$, $\gamma \in \Gamma$ is a \emph{generalised loxodromic element} if $\Gamma$ admits an acylindrical action on a hyperbolic space for which $\gamma$ is loxodromic.) As a consequence, the centraliser $C(g)$ of $g$ in $G \rtimes_\varphi H$ must contain $\langle g \rangle$ as a finite-index subgroup \cite[Corollary 6.6]{DGO}. By definition of $K$, there exists a map $\alpha : K \to G$ such that, for every $k \in K$, the automorphism $\varphi(k)$ coincides with the inner automorphism $\iota(\alpha(k))$. Notice that $k \alpha(k)^{-1}$ belongs to $C(g)$ for every $k \in K$. Indeed, 
$$k \alpha(k)^{-1} \cdot g \cdot \alpha(k) k^{-1} = k \cdot \alpha(k)^{-1}g\alpha(k) \cdot k^{-1}= \iota( \alpha(k)) \left(\alpha(k)^{-1}g\alpha(k) \right) = g.$$ 
Moreover, if $h,k \in K$ are such that $h \alpha(h)^{-1} \langle g \rangle = k \alpha(k)^{-1} \langle g \rangle$, then we have
$$k^{-1}h \in \alpha(k)^{-1} \langle g \rangle \alpha(h) \subset G,$$
hence $h=k$ since $K \cap G \subset H \cap G = \{1\}$. Consequently, because $\langle g \rangle$ has finite index in $C(g)$, it follows that $K$ must be finite.

\medskip \noindent
Conversely, suppose that $K$ is finite. We want to prove that $G \rtimes_\varphi H$ is acylindrically hyperbolic. First of all, notice that we have a short exact sequence
$$1 \to K \to G \rtimes_\varphi H \to G \rtimes_{\bar{\varphi}} (H/K) \to 1$$
where $\bar{\varphi} : H/K \to \mathrm{Aut}(G)$ is induced by $\varphi : H \to \mathrm{Aut}(G)$. Because being acylindrically hyperbolic is stable under quotients with finite kernels and extensions with finite kernels (according to \cite[Lemma 1]{ErratumMO}), it follows that $G \rtimes_\varphi H$ is acylindrically hyperbolic if and only if so is $G \rtimes_{\bar{\varphi}} (H/K)$. 

\medskip \noindent
In other words, we have reduced the problem to the case where $H$ is a subgroup $\mathrm{Aut}(G)$ intersecting trivially $\mathrm{Inn}(G)$. In the sequel, we assume that we are in this situation in order to simplify the notation. 

\medskip \noindent
Notice that $\langle \mathrm{Inn}(G), H \rangle = \mathrm{Inn}(G) \rtimes H$ since $\mathrm{Inn}(G)$ is a normal subgroup and $H \cap \mathrm{Inn}(G)= \{1\}$. On the other hand, because the center $Z(G)$ of $G$ is a characteristic normal subgroup, we have the short exact sequence
$$1 \to Z(G) \to G \rtimes H \to \left( G/Z(G) \right) \rtimes H \to 1.$$
By noticing that $\varphi \cdot \iota(g) \cdot \varphi^{-1}= \iota(\varphi(g))$ for every $g \in G$, we deduce that the canonical isomorphism 
$$\left\{ \begin{array}{ccc} G/Z(G) & \to & \mathrm{Inn}(G) \\ g & \mapsto & \iota(g) \end{array} \right.$$ 
induces an isomorphism $\left( G/Z(G) \right) \rtimes H \to \mathrm{Inn}(G) \rtimes H$. Because being acylindrically hyperbolic is stable under quotients with finite kernels and extensions with finite kernels (according to \cite[Lemma 1]{ErratumMO}), it follows that $G \rtimes H$ is acylindrically hyperbolic if and only if so is $\langle \mathrm{Inn}(G),H \rangle$. 

\medskip \noindent
So it remains to show that $\langle \mathrm{Inn}(G),H \rangle$ is acylindrically hyperbolic. Notice that, as a consequence of \cite[Lemma 7.2]{OsinAcyl}, any infinite normal subgroup of an acylindrically hyperbolic group must contain at least one of the generalised loxodromic elements of the group. Therefore, $\mathrm{Aut}(G)$ must contain a generalised loxodromic element which is an inner automorphism. It follows that $\langle \mathrm{Inn}(G),H \rangle$ contains a generalised loxodromic element in its own right, i.e., $\langle \mathrm{Inn}(G),H \rangle$ is acylindrically hyperbolic or virtually cyclic. But $\mathrm{Inn}(G)$ cannot be virtually cyclic because it is a normal subgroup of an acylindrically hyperbolic group, again according to \cite[Lemma 7.2]{OsinAcyl}, so the desired conclusion holds. 
\end{proof}

\begin{remark}
It is worth noticing that in the first part of the proof nothing is assumed about the center of $G$ nor the acylindrical hyperbolicity of its automorphism group. Actually, the argument shows more generally the following statement: If $1 \to N \to G \to Q \to 1$ is a short exact sequence with $N$ infinite and $G$ acylindrically hyperbolic, then $\mathrm{ker}(Q \to \mathrm{Out}(N))$ has to be finite. However, due to the absence of a natural morphism $Q \to \mathrm{Aut}(N)$, it is not clear how to embed $G$ into $\mathrm{Aut}(N)$ in order to deduce the acylindrical hyperbolicity of $G$ from the acylindrical hyperbolicity of $\mathrm{Aut}(N)$, as in the second part of the proof above. In other words, we do not know the answer to the following question: Given a short exact sequence $1 \to N \to G \to Q \to 1$, if $N$ has a finite center and if its automorphism group is acylindrically hyperbolic, is it true that $G$ is acylindrically hyperbolic if and only if $\mathrm{ker}(Q \to \mathrm{Out}(N))$ is finite? Theorem \ref{thm:ExtensionGen} provides a positive answer when the short sequence splits.
\end{remark}

\noindent
Because non-elementary hyperbolic groups have finite centers, our next statement follows immediately from Theorems \ref{thm:extension} and \ref{thm:main}:

\begin{cor}
Let $G$ be a one-ended hyperbolic group. Fix a group $H$ and a morphism $\varphi : H \to \mathrm{Aut}(G)$. The semidirect product $G \rtimes_\varphi H$ is acylindrically hyperbolic if and only~if $$K:= \mathrm{ker} \left( H \overset{\varphi}{\to} \mathrm{Aut}(H) \to \mathrm{Out}(H) \right)$$ is a finite subgroup of $H$.
\end{cor}

\noindent
In the specific case of surface groups, one gets:

\begin{cor}\label{cor:MCG}
Let $S$ be a closed orientable surface of genus $\geq 2$. Fix a point $p \in S$, a group $H$ and morphism $\varphi : H \to \mathrm{MCG}^{\pm}(S, p)$. The semidirect product $\pi_1(S,p) \rtimes_\varphi H$ is acylindrically hyperbolic if and only if $H \overset{\varphi}{\to} \mathrm{MCG}^{\pm}(S,p) \to \mathrm{MCG}^{\pm}(S)$ has finite image.
\end{cor}

\noindent
In this statement, $\mathrm{MCG}^\pm(S)$ (resp. $\mathrm{MCG}^\pm (S,p)$) refers to the \emph{extended mapping class group}, i.e., the group of diffeomorphisms of $S$ (resp. the group of diffeomorphisms of $S$ fixing $p$) up to isotopy (resp. up to isotopy fixing $p$). An orientation of $S$ is not assumed to be preserved. There clearly exists a morphism $\psi : \mathrm{MCG}^\pm(S,p) \to \mathrm{Aut}(\pi_1(S,p))$, and we denote by abuse of notation $G\rtimes_\varphi H$ the semidirect product $G\rtimes_{\psi \circ \varphi} H$. Finally, the morphism $f : \mathrm{MCG}^\pm(S,p) \to \mathrm{MCG}^\pm (S)$ is the natural morphism which just ``forgets'' the point $p$.

\begin{proof}[Proof of Corollary \ref{cor:MCG}.]
The corollary is a direct consequence of Theorem \ref{thm:ExtensionGen} applied to $\pi_1(S,p)$ combined with the following commutative diagram:
\[
\xymatrix{
1 \ar[r] & \mathrm{Inn}(\pi_1(S,p)) \ar[r] & \mathrm{Aut}(\pi_1(S,p)) \ar[r] & \mathrm{Out}(\pi_1(S,p)) \ar[r] & 1 \\
1  \ar[r] & \pi_1(S,p) \ar[u]^\simeq \ar[r] & \mathrm{Mod}^\pm(S,p) \ar[u]^\simeq_\psi \ar[r]^f & \mathrm{Mod}^\pm (S) \ar[u]^\simeq \ar[r] & 1
}
\]
\noindent
We refer to \cite[Page 235]{BookMCG} for more details.
\end{proof}

\addcontentsline{toc}{section}{References}

\bibliographystyle{alpha}
{\footnotesize\bibliography{AutHypAcyl}}

\def\polhk#1{\setbox0=\hbox{#1}{\ooalign{\hidewidth
  \lower1.5ex\hbox{`}\hidewidth\crcr\unhbox0}}}
\begin{thebibliography}{DGO17}

\bibitem[Abe79]{Abels}
H.~Abels.
\newblock An example of a finitely presented solvable group.
\newblock In {\em Homological group theory ({P}roc. {S}ympos., {D}urham,
  1977)}, volume~36 of {\em London Math. Soc. Lecture Note Ser.}, pages
  205--211. Cambridge Univ. Press, Cambridge-New York, 1979.

\bibitem[AD16]{Pnaive}
C.~Abbott and F.~Dahmani.
\newblock Property ${P}_\text{naive}$ for acylindrically hyperbolic groups.
\newblock {\em to appear in Math. Z., arxiv:1610.04143}, 2016.

\bibitem[BF95]{OutHyp}
M.~Bestvina and M.~Feighn.
\newblock Stable actions of groups on real trees.
\newblock {\em Invent. Math.}, 121(2):287--321, 1995.

\bibitem[Bow98]{BowditchJSJ}
B.~Bowditch.
\newblock Cut points and canonical splittings of hyperbolic groups.
\newblock {\em Acta Math.}, 180(2):145--186, 1998.

\bibitem[Bow08]{BowditchAcylMCG}
B.~Bowditch.
\newblock Tight geodesics in the curve complex.
\newblock {\em Invent. Math.}, 171(2):281--300, 2008.

\bibitem[Bra00]{FiniteHypSub}
N.~Brady.
\newblock Finite subgroups of hyperbolic groups.
\newblock {\em Internat. J. Algebra Comput.}, 10(4):399--405, 2000.

\bibitem[CDP90]{GeomHypGroupes}
M.~Coornaert, T.~Delzant, and A.~Papadopoulos.
\newblock {\em G\'{e}om\'{e}trie et th\'{e}orie des groupes}, volume 1441 of
  {\em Lecture Notes in Mathematics}.
\newblock Springer-Verlag, Berlin, 1990.
\newblock Les groupes hyperboliques de Gromov. [Gromov hyperbolic groups], With
  an English summary.

\bibitem[Cha95]{ChampetierStat}
C.~Champetier.
\newblock Propri\'{e}t\'{e}s statistiques des groupes de pr\'{e}sentation
  finie.
\newblock {\em Adv. Math.}, 116(2):197--262, 1995.

\bibitem[CM15]{CashenManning}
C.~Cashen and J.~Manning.
\newblock Virtual geometricity is rare.
\newblock {\em LMS J. Comput. Math.}, 18(1):444--455, 2015.

\bibitem[CU18]{Simul}
M.~Clay and C.~Uyanik.
\newblock Simultaneous construction of hyperbolic isometries.
\newblock {\em Pacific J. Math.}, 294(1):71--88, 2018.

\bibitem[CV96]{FAmcg}
M.~Culler and K.~Vogtmann.
\newblock A group-theoretic criterion for property {${\rm FA}$}.
\newblock {\em Proc. Amer. Math. Soc.}, 124(3):677--683, 1996.

\bibitem[DG11]{HypIso}
F.~Dahmani and V.~Guirardel.
\newblock The isomorphism problem for all hyperbolic groups.
\newblock {\em Geom. Funct. Anal.}, 21(2):223--300, 2011.

\bibitem[DGO17]{DGO}
F.~Dahmani, V.~Guirardel, and D.~Osin.
\newblock Hyperbolically embedded subgroups and rotating families in groups
  acting on hyperbolic spaces.
\newblock {\em Mem. Amer. Math. Soc.}, 245(1156):v+152, 2017.

\bibitem[FM12]{BookMCG}
B.~Farb and D.~Margalit.
\newblock {\em A primer on mapping class groups}, volume~49 of {\em Princeton
  Mathematical Series}.
\newblock Princeton University Press, Princeton, NJ, 2012.

\bibitem[Gen18]{AcylAutRAAG}
A.~Genevois.
\newblock Negative curvature of automorphism groups of graph products with
  applications to right-angled {A}rtin groups.
\newblock {\em arXiv:1807.00622}, 2018.

\bibitem[Gho18]{FreeByCyclicAcyl}
P.~Ghosh.
\newblock Acylindrical hyperbolicity of free-by-cyclic extensions.
\newblock {\em arXiv:1802.08570}, 2018.

\bibitem[GM18]{GM}
A.~Genevois and A.~Martin.
\newblock Automorphisms of graph products of groups from a geometric
  perspective.
\newblock {\em arXiv:1809.08091}, 2018.

\bibitem[Gro87]{GromovHyp}
M.~Gromov.
\newblock Hyperbolic groups.
\newblock In {\em Essays in group theory}, volume~8 of {\em Math. Sci. Res.
  Inst. Publ.}, pages 75--263. Springer, New York, 1987.

\bibitem[Gui08]{GuirardelTrees}
V.~Guirardel.
\newblock Actions of finitely generated groups on {$\Bbb R$}-trees.
\newblock {\em Ann. Inst. Fourier (Grenoble)}, 58(1):159--211, 2008.

\bibitem[KB02]{BoundaryHyp}
I.~Kapovich and N.~Benakli.
\newblock Boundaries of hyperbolic groups.
\newblock In {\em Combinatorial and geometric group theory ({N}ew {Y}ork,
  2000/{H}oboken, {NJ}, 2001)}, volume 296 of {\em Contemp. Math.}, pages
  39--93. Amer. Math. Soc., Providence, RI, 2002.

\bibitem[KKN18]{AutFreeTGeneral}
M.~Kaluba, D.~Kielak, and P.~Nowak.
\newblock On property ({T}) for $\mathrm{Aut}(\mathbb{F}_n)$ and
  $\mathrm{SL}_n(\mathbb{Z})$.
\newblock {\em arXiv:1812.03456}, 2018.

\bibitem[KL95]{KapovichLeebCones}
M.~Kapovich and B.~Leeb.
\newblock On asymptotic cones and quasi-isometry classes of fundamental groups
  of {$3$}-manifolds.
\newblock {\em Geom. Funct. Anal.}, 5(3):582--603, 1995.

\bibitem[KNO17]{AutFreeT}
M.~Kaluba, P.~Nowak, and N.~Ozawa.
\newblock $\mathrm{Aut}(\mathbb{F}_5)$ has property ({T}).
\newblock {\em arXiv:1712.07167}, 2017.

\bibitem[Kra81]{KraMCG}
I.~Kra.
\newblock On the {N}ielsen-{T}hurston-{B}ers type of some self-maps of
  {R}iemann surfaces.
\newblock {\em Acta Math.}, 146(3-4):231--270, 1981.

\bibitem[KS96]{QuasiconvexHyp}
I.~Kapovich and H.~Short.
\newblock Greenberg's theorem for quasiconvex subgroups of word hyperbolic
  groups.
\newblock {\em Canad. J. Math.}, 48(6):1224--1244, 1996.

\bibitem[Led18]{AutFreeProductFA}
N.~Leder.
\newblock Serre's {P}roperty {FA} for automorphism groups of free products.
\newblock {\em arXiv:1810.06287}, 2018.

\bibitem[Lev05]{LevittOut}
G.~Levitt.
\newblock Automorphisms of hyperbolic groups and graphs of groups.
\newblock {\em Geom. Dedicata}, 114:49--70, 2005.

\bibitem[MNS99]{MillerJSJ}
C.~Miller, III, W.~Neumann, and G.~Swarup.
\newblock Some examples of hyperbolic groups.
\newblock In {\em Geometric group theory down under ({C}anberra, 1996)}, pages
  195--202. de Gruyter, Berlin, 1999.

\bibitem[MO15]{MinasyanOsin}
A.~Minasyan and D.~Osin.
\newblock Acylindrical hyperbolicity of groups acting on trees.
\newblock {\em Math. Ann.}, 362(3-4):1055--1105, 2015.

\bibitem[MO17]{ErratumMO}
A.~Minasyan and D.~Osin.
\newblock Erratum to the paper ``{A}cylindrical hyperbolicity of groups acting
  on trees''.
\newblock {\em arXiv:1711.09486}, 2017.

\bibitem[OH07]{Ould}
A.~Ould~Houcine.
\newblock Embeddings in finitely presented groups which preserve the center.
\newblock {\em J. Algebra}, 307(1):1--23, 2007.

\bibitem[Osi16]{OsinAcyl}
D.~Osin.
\newblock Acylindrically hyperbolic groups.
\newblock {\em Trans. Amer. Math. Soc.}, 368:851--888, 2016.

\bibitem[Osi17]{OsinSurvey}
D.~Osin.
\newblock Groups acting acylindrically on hyperbolic spaces.
\newblock {\em arxiv:1712.00814}, 2017.

\bibitem[Pau91]{PaulinOut}
F.~Paulin.
\newblock Outer automorphisms of hyperbolic groups and small actions on {${\bf
  R}$}-trees.
\newblock In {\em Arboreal group theory ({B}erkeley, {CA}, 1988)}, volume~19 of
  {\em Math. Sci. Res. Inst. Publ.}, pages 331--343. Springer, New York, 1991.

\bibitem[PS17]{MCGacyl}
P.~Przytycki and A.~Sisto.
\newblock A note on acylindrical hyperbolicity of mapping class groups.
\newblock In {\em Hyperbolic geometry and geometric group theory}, volume~73 of
  {\em Adv. Stud. Pure Math.}, pages 255--264. Math. Soc. Japan, Tokyo, 2017.

\bibitem[Roe03]{Roe}
J.~Roe.
\newblock {\em Lectures on coarse geometry}, volume~31 of {\em University
  Lecture Series}.
\newblock American Mathematical Society, Providence, RI, 2003.

\bibitem[Sel97]{SelaJSJ}
Z.~Sela.
\newblock Structure and rigidity in ({G}romov) hyperbolic groups and discrete
  groups in rank {$1$} {L}ie groups. {II}.
\newblock {\em Geom. Funct. Anal.}, 7(3):561--593, 1997.

\bibitem[Ser03]{MR1954121}
J.-P. Serre.
\newblock {\em Trees}.
\newblock Springer Monographs in Mathematics. Springer-Verlag, Berlin, 2003.
\newblock Translated from the French original by John Stillwell, Corrected 2nd
  printing of the 1980 English translation.

\bibitem[Var14]{AutFreeFAn}
O.~Varghese.
\newblock Fixed points for actions of {${\rm Aut}(F_n)$} on {$\rm CAT(0)$}
  spaces.
\newblock {\em M\"{u}nster J. Math.}, 7(2):439--462, 2014.

\bibitem[Var18]{AutUnivCoxFAn}
O.~Varghese.
\newblock Automorphism group of universal {C}oxeter group.
\newblock {\em arXiv:1805.06748}, 2018.

\bibitem[Whi36]{Whitehead}
J.~Whitehead.
\newblock On equivalent sets of elements in a free group.
\newblock {\em Ann. of Math. (2)}, 37(4):782--800, 1936.

\end{thebibliography}

\Address

\end{document}